\documentclass{article}
\usepackage[utf8]{inputenc}
\usepackage{amsthm,amsmath,amssymb}
\usepackage{hyperref}

\title{Codimension one regular foliations on rationally connected threefolds}
\author{Jo\~ao Paulo Figueredo}
\date{October 2021}

\newtheorem{thm}{Theorem}[section]
\newtheorem{conj}[thm]{Conjecture}
\newtheorem{lem}[thm]{Lemma}
\newtheorem{cor}[thm]{Corollary}
\newtheorem{prop}[thm]{Proposition}
\theoremstyle{definition}
\newtheorem{defi}[thm]{Definition}
\newtheorem{rmk}[thm]{Remark}

\begin{document}

\maketitle

\begin{abstract}
    In his work on birational classification of foliations on projective surfaces, Brunella showed that every regular foliation on a rational surface is algebraically integrable with rational leaves. This led Touzet to conjecture that every regular foliation on a rationally connected manifold is algebraically integrable with rationally connected leaves. Druel proved this conjecture for the case of weak Fano manifolds. In this paper, we extend this result showing that Touzet's conjecture is true for codimension one foliations on threefolds with nef anti-canonical bundle.

\vspace{0.5cm}

\noindent\textbf{Keywords}: Regular foliation, holomorphic foliation, foliated MMP, rationally connected manifolds \\ \textbf{Classification}: 14M22 (Primary), 37F75 (Secondary)
\end{abstract}

\section{Introduction}

A singular holomorphic foliation on a normal complex variety $X$ is defined by a coherent subsheaf $\mathcal{F} \subset T_X$ which is closed under the Lie bracket and such that $T_X/\mathcal{F}$ is torsion free. There exists an open subset $X_0$ of the regular locus of $X$ over which $\mathcal{F}$ is a subbundle of $T_X$. The rank of $\mathcal{F}$ is defined as the rank of $\mathcal{F}_{|X_0}$. We say that $\mathcal{F}$ is a regular foliation if $X_0 = X$.

By the classical Frobenius integrability theorem (theorem \ref{thm:frobenius}), $\mathcal{F}$ defines a decomposition of $X_0$ into a disjoint union of complex immersed submanifolds, having dimension equal to the rank of $\mathcal{F}$, such that locally analytically around each point $x \in X_0$, these submanifolds are the fibers of a submersion $f\colon U \rightarrow V$, where $\dim(V) = \dim(X) - r$. These immersed submanifolds are called the leaves of the foliation $\mathcal{F}$. We will say that $\mathcal{F}$ is algebraically integrable if for every leaf $F \subset X$ of $\mathcal{F}$, we have $\dim(F) = \dim(\overline{F}^{\rm Zar})$, where $\overline{F}^{\rm Zar}$ is the Zariski closure of $F$ in $X$.

As an example of a foliation, consider a vector field $v \in H^0(X,T_X)$, vanishing in at most a finite set of points. Then the sheaf generated by $v$ defines a foliation of rank one on $X$. The leaves are obtained by solving the ordinary differential equation $\gamma'(t) = v(\gamma(t))$. Thus, to determine whether a foliation of rank one is algebraically integrable is equivalent to determine whether the solutions of an ordinary differential equation are algebraic.

Naturally associated to a foliation $\mathcal{F}$ there is a canonical sheaf $K_\mathcal{F} = \det(\mathcal{F}^*)$. Conjecturally, the numerical properties of the canonical sheaf of a foliation $\mathcal{F}$ governs the geometry of $\mathcal{F}$, just like the canonical sheaf of a projective variety $X$ governs the geometry of $X$. In particular, there is a conjectural MMP for foliations (using $K_\mathcal{F}$ in the place of $K_X$), currently established if $\dim(X) \leq 3$ (\cite{brunella1999minimal}, \cite{mcquillan}, \cite{Spicer2017}, \cite{Cascini2018}, \cite{spicer2019local}, \cite{cascini2020mmp}). 

This foliated Minimal Model Program is significantly simplified for regular foliations of codimension $1$ (\cite{Spicer2017}). If $\dim(X) = 2$ and $\mathcal{F}$ is a regular foliation of rank $1$ on $X$, then either $K_\mathcal{F}$ is nef or $\mathcal{F}$ is induced by a rational fibration (hence algebraically integrable). If $\dim(X) = 3$ and $\mathcal{F}$ a regular foliation of codimension $1$ on $X$, then after a sequence of smooth blow-downs we obtain a pair $(Y,\mathcal{G})$ where $Y$ and $\mathcal{G}$ are regular, and either $K_\mathcal{G}$ is nef, or $Y$ is a Mori Fiber Space with $K_\mathcal{G}$-negative fibers tangent to $\mathcal{G}$.

In this paper we use this MMP to classify regular foliations of codimension $1$ on $3$-folds. To illustrate this method, let $X$ be a rational surface. If $\mathcal{F}$ is not algebraically integrable, then $K_\mathcal{F}$ is nef. By Bott's vanishing theorem (see theorem \ref{thm:baumbott}), we have $c_1(T_X/\mathcal{F})^2 = c_2(T_X/\mathcal{F}) = 0$, and this implies that $c_1(T_X)^2 - c_2(T_X) \geq 0$. Thus  $X$ has to be a Hirzebruch surface, and we conclude the following:

\begin{thm}[\cite{brunellasurfaces}]
Let $X$ be a projective rational surface and let $\mathcal{F}$ be a regular foliation of rank $1$ on $X$. Then $\mathcal{F}$ is induced by a smooth morphism with rational fibers. \label{thm:brunella}
\end{thm}

It is then natural to wonder if a similar result holds for manifolds of higher dimension. This led Touzet to make the following conjecture (see \cite{druel2017regular}):

\begin{conj}[Touzet]
Let $X$ be a projective rationally connected manifold and let $\mathcal{F}$ be a regular foliation on $X$. Then $\mathcal{F}$ is induced by a smooth morphism with rationally connected fibers.
\end{conj}

This conjecture is open for $\dim(X) \geq 3$. By a well known result of Campana and Kollar-Miyaoka-Mori, every weak Fano manifold is rationally connected. Druel showed the following theorem, which confirms Touzet's conjecture in the case of weak Fanos:

\begin{thm}[\cite{druel2017regular}]
Let $X$ be weak Fano manifold, i.e. $-K_X$ is nef and big. Let $\mathcal{F}$ be a regular foliation on $X$. Then $\mathcal{F}$ is induced by a smooth morphism with rationally connected fibers.
\end{thm}

In this paper we address Touzet's conjecture in the case when $\dim(X) = 3$ and $\mathcal{F}$ has codimension $1$. Our main result is the following, which generalizes Druel's result in the case of codimension one foliations on threefolds:

\begin{thm}
Let $X$ be a rationally connected threefold with $-K_X$ nef. Let $\mathcal{F}$ be a regular foliation of codimension $1$ on $X$. Then $\mathcal{F}$ is induced by a smooth morphism with rational fibers. \label{thm:main}
\end{thm}

To prove this result, we apply the foliated MMP to $\mathcal{F}$. When $K_\mathcal{F}$ is not pseudo-effective, we end up with a pair $(Y,\mathcal{G})$ where $Y$ and $\mathcal{G}$ are regular and $Y$ is a Mori Fiber Space with fibers tangent to $\mathcal{G}$. By using Mori's classification of smooth Mori Fiber Spaces of dimension $3$, we show directly that the foliation $\mathcal{G}$ is induced by a smooth morphism with rational leaves.

When $K_\mathcal{F}$ is pseudo-effective, we end up with a pair $(Y,\mathcal{G})$ where $Y$ and $\mathcal{G}$ are regular and $K_\mathcal{G}$ is nef. We then show that $-K_Y$ is nef, and apply results of Bauer-Peternell, which classify the nef reduction map of $-K_Y$ in the case of rationally connected threefolds. We are able to show that this map is actually a fibration by $K3$-surfaces, and induces $\mathcal{G}$. We will show that this is a contradiction. Hence $K_\mathcal{F}$ cannot be pseudo-effective, and $\mathcal{F}$ is induced by a smooth morphism with rational fibers.

\subsubsection*{Acknowledgments}
The author would like to thank St\'ephane Druel and Carolina Araujo for the help he received in the preparation of his thesis, whose second part gives this paper. During the writing of this paper, the author received financial support from CNPq (process number 140605/2017-7), and from CAPES/COFECUB (process number 88887.192325/2018-00 / Ma 932/19).

\section{Preliminaries}
In this section we define foliations and discuss some of  its properties. In particular, we state the foliated Minimal Model Program for regular foliations of codimension one on threefolds, proved by Spicer. We also collect some technical results on nef reduction maps, isotrivial families and Mori contractions in the presence of foliations, which will be later used in the proof of our main result.

\subsubsection*{Holomorphic foliations}

We begin with the basic definitions and results which we will use concerning holomorphic foliations.

\begin{defi}
A holomorphic foliation of rank $r$ (or codimension $n-r$) on a normal complex variety $X$ of dimension $n$ is defined by a coherent subsheaf $\mathcal{F} \subset T_X$, with generic rank $r$, such that $[\mathcal{F},\mathcal{F}] \subset \mathcal{F}$ and $T_X/\mathcal{F}$ is torsion free. The canonical sheaf of $\mathcal{F}$ is defined as $K_\mathcal{F} = \det(\mathcal{F}^*)$. The singular locus of $(X,\mathcal{F})$ is defined as ${\rm sing}(\mathcal{F}) = X \setminus X_0$, where $X_0$ is the open subset of the regular locus of $X$ where $\mathcal{F}$ is a subvector-bundle of $T_X$. We say $\mathcal{F}$ is a regular foliation if $X$ is regular and $\mathcal{F}$ is a subvector-bundle of $T_X$ (i.e. ${\rm sing}(\mathcal{F}) = \emptyset$).
\end{defi}

The following theorem implies that $\mathcal{F}$ decomposes $X_0$ into a union of immersed submanifolds of dimension $r = {\rm rank}(\mathcal{F})$ which will be called the leaves of $\mathcal{F}$.

\begin{thm}[{{\cite{clebsch1866ueber}}}]
Let $X$ be a complex manifold of dimension $n$ and $T \subset T_X$ a subbundle of rank $r$. Suppose that for any two sections $v,w$ of $T$, we have $[v,w]$ a section of $T$. Then, for any point $p \in X$, there exists an open neighborhood $U$ of $p$ in $X$, such that $U \cong \mathbb{D}^r \times V$, where $V$ is a manifold of dimension $n-r$ and $\mathbb{D} \subset \mathbb{C}$ is the open unitary disk, such that $T_{|U} \cong \ker(d\pi)$, where $\pi\colon \mathbb{D}^r \times V \rightarrow V$ is the second projection.\label{thm:frobenius}
\end{thm}

\begin{defi}
Let $X$ be a normal complex variety and let $\mathcal{F}$ be a foliation of rank $r$ on $X$. Let $X_0$ be the open subset of the regular locus of $X$ such that $\mathcal{F}$ is a subvector-bundle of $T_X$ on $X_0$. We say that a immersed submanifold $F \subset X_0$ of dimension $r$ is a leaf of $\mathcal{F}$ if $T_F \rightarrow (T_X)_{|F}$ factors through $\mathcal{F}_{|F} \rightarrow (T_X)_{|F}$.
\end{defi}

\begin{rmk}
We might also define foliations by twisted reflexive differential forms. Indeed, let $X_0$ be as above. On $X_0$ we have an exact sequence of vector-bundles:
\[0 \rightarrow \mathcal{F}_{|X_0} \rightarrow (T_X)_{|X_0} \rightarrow (T_X/\mathcal{F})_{|X_0}\rightarrow 0.\]
Dualizing $(T_X)_{|X_0} \rightarrow (T_X/\mathcal{F})_{|X_0}\rightarrow 0$, we get an injective morphism
\[ 0 \rightarrow (T_X/\mathcal{F})_{|X_0}^* \rightarrow (\Omega_X^1)_{|X_0},\]
and thus a global section 
\[ 0 \rightarrow \mathcal{O}_{X_0} \rightarrow (\Omega_X^1)_{|X_0} \otimes (T_X/\mathcal{F})_{|X_0},\]
i.e. a non-zero section $\omega \in H^0(X_0,(\Omega_X^1)_{|X_0} \otimes (T_X/\mathcal{F})_{|X_0})$. The sheaf $T_X/\mathcal{F}$ is called the normal sheaf of $\mathcal{F}$, and we will denote it by $N_\mathcal{F}$.
\end{rmk}

The previous remark allows us to define the pullback of foliations by dominant maps:

\begin{rmk}[{{\cite[3.2]{druelcodim1}}}]
Let $\pi\colon Y \dashrightarrow X$ be a dominant rational map between normal varieties. Let $\mathcal{F}$ be a codimension $q$ foliation on $X$. Then there exists a codimension $q$ foliation $\pi^{-1}(\mathcal{F})$ on $Y$, called the pullback of $\mathcal{F}$ by $\pi$, such that the leaves of $\pi^{-1}(\mathcal{F})$ are pre-images of leaves of $\mathcal{F}$.
\end{rmk}

We see that if we pullback a foliation by a morphism $\pi$, then the general fiber of $\pi$ is tangent to this pullback foliation. The following lemma shows that the converse is also true, i.e., that if there is a fibration $\pi$ with general fiber tangent to $\mathcal{F}$, then $\mathcal{F}$ is the pullback of a foliation by $\pi$.

\begin{lem}[{{\cite[Lemma 6.7]{fanofoliations}}}]\label{lem:pullbackbyequidimensional}
Let $\pi\colon X \rightarrow Y$ be an equidimensional morphism of connected fibers between normal varieties. Let $\mathcal{F}$ be a foliation of rank $r$ on $X$. Suppose the general fiber of $\pi$ is tangent to $\mathcal{F}$. Then there exists a foliation $\mathcal{G}$ on $Y$, of rank $r - (\dim(X) - \dim(Y))$, such that the following sequence is exact:
\[0 \rightarrow T_{X/Y} \rightarrow \mathcal{F} \rightarrow (\pi^*\mathcal{G})^{**}.\]
In this case, $\mathcal{F} = \pi^{-1}(\mathcal{G})$.
\end{lem}

The importance of the normal bundle for regular foliations will come from the following two results. They will be essential in the proof of our main result.

\begin{rmk}\label{rmk:bottconnection}
Let $X$ be a manifold and let $\mathcal{F}$ be a regular foliation on $X$. Consider $N_\mathcal{F}$ the normal bundle of $\mathcal{F}$. Then there is a natural $\mathcal{F}$-connection $\nabla$ on $N_\mathcal{F}$. Indeed, let $U$ be a local section of $N_\mathcal{F}$ and let $V$ be a local section of $\mathcal{F}$. Then we may define
\[\nabla_V(U) = p([V,T]),\]
where $p\colon T_X \rightarrow N_\mathcal{F}$ is the projection and $T$ is a local section of $T_X$ such that $p(T) = U$. Since $\mathcal{F}$ is closed under the Lie bracket, this is well defined.

Now, let $Z \subset X$ be a submanifold which is tangent to $\mathcal{F}$, i.e. $T_Z \subset (\mathcal{F})_{|Z}$. Then, $\nabla_{|Z}$ is a holomorphic connection on ${N_\mathcal{F}}_{|Z}$. In particular, all the Chern classes of ${N_\mathcal{F}}_{|Z}$ vanish (see \cite[Theorem 4]{atiyah1957complex}).
\end{rmk}

Before stating the other result, let us define a Chern polynomial of degree $d$ of a vector bundle $E$ on a manifold $X$ as any element of the form $c_{i_1}(E)\cdot c_{i_2}(E) \cdot \dots c_{i_s}(E) \in H^{2i_1+2i_2+\dots+2i_s}(X,\mathbb{C})$, where $d = i_1 +\dots + i_s$. 

\begin{thm}[{{\cite{baumbott}}}]\label{thm:baumbott}
Let $X$ be a manifold and let $\mathcal{F}$ be a codimension $q$ foliation on $X$. Let $\varphi = \varphi(N_\mathcal{F}) \in H^{2l}(X,\mathbb{C})$ be a Chern polynomial on $N_\mathcal{F}$ of degree $l$, with $q < l \leq \dim(M)$. If $\mathcal{F}$ is regular, then $\varphi(N_\mathcal{F}) = 0$.
\end{thm}

\begin{rmk}\label{rmk:intersectionwithrationalcurve}
Let $X$ be a manifold and $\mathcal{F}$ a regular codimension $1$ foliation on $M$. Let $C \subset X$ be a smooth proper curve on $X$. By remark \ref{rmk:bottconnection}, if $C$ is tangent to $\mathcal{F}$, then $c_1(N_\mathcal{F}) \cdot C = 0$. Now suppose $C$ is not tangent to $\mathcal{F}$. Let 
\[\omega \in H^0(X,\Omega_X^1 \otimes N_\mathcal{F})\]
define $\mathcal{F}$. Then, since $C$ is not tangent to $\mathcal{F}$,
\[\omega_{|C} \in H^0(C,\Omega_C^1 \otimes {N_\mathcal{F}}_{|C})\]
is non-zero. This implies that $\Omega_C^1 \otimes {N_\mathcal{F}}_{|C}$ is a vector bundle on $C$ with non negative degree. Since $\deg(\Omega_C^1) = 2g(C) - 2$, where $g(C)$ is the genus of $C$, we conclude that if $C$ is not tangent to $\mathcal{F}$, then $c_1(N_\mathcal{F})\cdot C \geq 2 - 2g(C)$.

In particular, if $g(C) = 0$, then $c_1(N_\mathcal{F}) \cdot C \geq 0$ and equality holds iff $C$ is tangent to $\mathcal{F}$.
\end{rmk}
\subsubsection*{Foliated Minimal Model Program}
We state the foliated Minimal Model Program for regular codimension one foliations on $3$-folds. 

\begin{thm}[\cite{Spicer2017}]
Let $X$ be a projective manifold of dimension $3$. Let $\mathcal{F}$ be a codimension one foliation on $X$. Then there is a sequence of smooth blow-ups centered at smooth curves:
\[X \rightarrow X_1 \rightarrow \dots \rightarrow X_n,\]
such that if $\mathcal{F}_i$ is the foliation induced by $\mathcal{F}$ on $X_i$, then $\mathcal{F}_i$ is regular for each $i$, and the non-trivial fibers of $X_i \rightarrow X_{i+1}$ are tangent to $\mathcal{F}_i$. Moreover one the following holds:
\begin{itemize}
    \item[(a)] Either $K_{\mathcal{F}_n}$ is nef;
    \item[(b)] Or there exists a structure of Mori Fiber Space $X_n \rightarrow Y$, whose fibers are tangent to $\mathcal{F}_n$.
\end{itemize}\label{thm:spicer}
\end{thm}

In the case (b) of the theorem above, we have a classification by Mori of the possibilities for $X_n \rightarrow Y$:

\begin{thm}[{{\cite[Theorem 3.3, Theorem 3.5]{mori}}}]
Let $X$ be a projective manifold of dimension $3$. Let $R \subset \overline{NE}(X)$ be a $K_X$-negative extremal ray, and consider $\varphi\colon X \rightarrow Y$ the contraction associated to $R$. Then either $\varphi$ contracts an irreducible divisor $D$ or $\varphi$ is a fibration by Fano varieties and $Y$ is smooth. Moreover, we have the following cases:
\begin{enumerate}
    \item If $\varphi$ is divisorial, then one of the following holds:
    \begin{itemize}
        \item[(a)] $\varphi$ is the blow-up of a smooth curve $C$ on $Y$ and $Y$ is smooth;
        \item[(b)] $Q=\varphi(D)$ is a point, $Y$ is smooth, $D \cong \mathbb{P}^2$ and $\mathcal{O}_D(D) \cong \mathcal{O}_\mathbb{P}(-1)$;
        \item[(c)] $Q=\varphi(D)$ is a point, $D \cong \mathbb{P}^1 \times \mathbb{P}^1$, $\mathcal{O}_D(D)$ is of degree $(-1,-1)$ and $s \times \mathbb{P}^1 \sim \mathbb{P}^1 \times t$ on $X$ ($s,t \in \mathbb{P}^1$);
        \item[(d)] $Q=\varphi(D)$ is a point, $D$ is isomorphic to an irreducible reduced singular quadric in $\mathbb{P}^3$, $\mathcal{O}_D(D) \cong \mathcal{O}_D \otimes \mathcal{O}_\mathbb{P}(-1)$; or
        \item[(e)] $Q=\varphi(D)$ is a point, $D \cong \mathbb{P}^2$ and $\mathcal{O}_D(D) \cong \mathcal{O}_\mathbb{P}(-2)$.
    \end{itemize}
    \item If $\varphi$ is a Mori fibration, then $Y$ is smooth and one of the following holds:
    \begin{itemize}
        \item[(a)] $\dim(Y) = 2$ and for every $y \in Y$, $\varphi^{-1}(y)$ is a conic in $\mathbb{P}^2$;
        \item[(b)] $\dim(Y) = 1$ and $-K_X$ is relatively ample; or
        \item[(c)] $\dim(Y) = 0$ and $X$ is Fano.
    \end{itemize}
    Moreover in case (a), the discriminant $\Delta = \{y \in Y \mid \varphi^{-1}(y) \text{ is singular}\}$ of $\varphi$ has only ordinary double points as singularities, and $\varphi^{-1}(y)$ is a double line iff $y$ is a singular point of $\Delta$. In case (b), we have $1 \leq (K_{\varphi^{-1}(y)})^2 \leq 6$ or $(K_{\varphi^{-1}(y)})^2 = 8,9$; if $(K_{\varphi^{-1}(y)})^2 = 9$, then $\varphi$ is a $\mathbb{P}^2$-bundle; if $(K_{\varphi^{-1}(y)})^2 = 8$, then $X$ is embedded in a $\mathbb{P}^3$-bundle $P$ over $Y$ such that, for all $y \in Y$, $\varphi^{-1}(y)$ is an irreducible reduced quadric of $\mathbb{P}^3$.
\end{enumerate} \label{thm:mori}
\end{thm}

\subsubsection*{Nef reduction of anticanonical bundles}

Another tool we will use in our proof is the classification of the nef reduction map of the anticanonical bundle of rationally connected threefolds with nef anticanonical bundle, given by Bauer and Peternell. 

\begin{thm}[{{\cite[Theorem 2.1]{bauer2004}}}]
Let $X$ be a projective threefold with $-K_X$ nef. Then there exists a morphism $f\colon X \rightarrow B$ to a normal projective variety $B$ such that 
\begin{enumerate}
\item $-K_X$ is numerically trivial on all fibers of $f$;
\item for $x \in X$ general and every irreducible curve $C$ passing through $x$ such that $\dim f(C) > 0$, we have $-K_X \cdot C > 0$.
\end{enumerate}
\label{thm:3fold}
\end{thm}

This result is more precise if $X$ is rationally connected with $K_X^2 \equiv 0$. We will show that this condition $K_X^2 \equiv 0$ always happens if $X$ admits a regular foliation of codimension one with nef canonical bundle. Moreover we will only be concerned with the cases (1) and (2a) of theorem \ref{thm:mori}. Under these conditions, we have the following theorem:

\begin{thm}[{{\cite[Corollary 2.4]{bauer2004}}}]
Let $X$ be a smooth rationally connected  projective threefold with $-K_X$ nef. Suppose that $K_X^2 \equiv 0$ and that  $X$ is as in cases (1) or (2a) of theorem \ref{thm:mori}. Then $-K_X$ induces a $K3$-fibration $f\colon X \rightarrow \mathbb{P}^1$ and $-K_X \cong f^*(\mathcal{O}_{\mathbb{P}^1}(1))$. \label{thm:dimnum1}
\end{thm}

\subsubsection*{Isotrivial families}

In the proof of our main result, we will need to show that, under certain conditions, the $K3$ fibration obtained in theorem \ref{thm:dimnum1} is isotrivial. We begin with the definition:

\begin{defi}
Let $f\colon X \rightarrow Y$ be a surjective morphism between normal projective varieties with connected general fiber $F$. We say $f$ is birationally isotrivial if $X \times_Y {\rm Spec} {\overline{\mathbb{C}(Y)}}$ is birational to $F \times {\rm Spec} {\overline{\mathbb{C}(Y)}}$, where $\mathbb{C}(Y)$ is the field of rational functions on $Y$. 
\end{defi}

The first result shows that under some conditions, a smooth family over $\mathbb{P}^1$ is isotrivial:

\begin{thm}[{{\cite[Theorem 0.1]{viehwegzuo}}}]\label{thm:viehwegzuo}
Suppose $f\colon X \rightarrow \mathbb{P}^1$ is a surjective morphism with connected general fiber $F$, where $X$ is a projective manifold. Suppose that $f$ is not birationally isotrivial, and that $F$ has a minimal model $F'$ with $K_{F'}$ semi-ample. Then $f$ has at least  three singular fibers.
\end{thm}

The second result shows that under some conditions, an isotrivial family of surfaces over a curve is trivial after an \'etale change of base:

\begin{lem}[{{\cite[Lemma 1.6]{oguisoviehweg}}}]\label{lem:oguisoviehweg}
Let $f\colon X \rightarrow C$ be a smooth projective family of minimal surfaces of non-negative Kodaira dimension. Then $f$ is birationally isotrivial if, and only if, there exists a finite étale cover $C'\rightarrow C$ and a surface $F$ with
\[X \times_C C' \cong F\times C'.\]
\end{lem}

\subsubsection*{Mori contractions in the presence of a codimension one regular foliation}

Finally, we will need to see how certain Mori contractions (as in theorem \ref{thm:mori}) behave in the presence of a regular foliation of codimension one. The following lemma shows that if we have a divisorial contraction to a point on $X$, then there is not regular foliation of codimension $1$ on $X$.

\begin{lem}\label{lem:divtopoint}
Let $X$ be a smooth $3$-fold and $\pi \colon X \rightarrow Y$ one of the divisorial contractions to a point in theorem \ref{thm:mori}. If $\mathcal{F}$ is a codimension $1$ foliation on $X$, then ${\rm sing}(\mathcal{F}) \neq \emptyset$ 
\end{lem}
\begin{proof}
Denote by $E$ the contracted divisor of $\pi$. Then, if $\ell \subset E$ is a rational curve contracted by $\pi$, then $N_{E/X}\cdot \ell < 0$. By remark \ref{rmk:bottconnection}, this implies that if $E$ is tangent to $\mathcal{F}$, then ${\rm sing}(\mathcal{F}) \neq \emptyset$. We may then suppose that $E$ is not tangent to $\mathcal{F}$.

By theorem \ref{thm:mori}, we have three possibilities: $E \cong \mathbb{P}^2$, or $E \cong \mathbb{P}^1 \times \mathbb{P}^1$, or $E$ is isomorphic to a quadric cone in $\mathbb{P}^3$. Suppose that $\mathcal{F}$ is regular. Let $\omega \in H^0(X,\Omega_X^1 \otimes N_\mathcal{F})$ define $\mathcal{F}$. 

If $E \cong \mathbb{P}^2$ or a quadric cone, then we claim that $\omega_{|E} \equiv 0$. Indeed, consider $(N_\mathcal{F})_{|E} \in {\rm Pic}(E) \cong \mathbb{Z}$. Then, there is a line $\ell$ in $E$ such that $(N_\mathcal{F})_{|E} \sim_\mathbb{Q} a\ell$, for some $a \in \mathbb{Q}$. Since $\mathcal{F}$ is regular, we have $N_\mathcal{F}^2 \equiv 0$ and thus $a^2\ell^2 \sim 0$. This can only happen if $a = 0$. We conclude that $(N_\mathcal{F})_{|E} \cong \mathcal{O}_E$. Therefore $\omega_{|E} \in H^0(E,\Omega_E^1)$. In both cases of $E$, we have $H^0(E,\Omega_E^1) = 0$. Thus $\omega_{|E} \equiv 0$ and $E$ is invariant by $\mathcal{F}$, a contradiction.

If $E \cong \mathbb{P}^1 \times \mathbb{P}^1$, then by lemma \ref{lem:hirzsurface} below, $\mathcal{F}_{|E}$ is given by one of the projections $p\colon E \rightarrow \mathbb{P}^1$. In particular $N_\mathcal{F} \cdot \ell = 0$ for every fiber $\ell$ of $p$. Now, if $f$ is a fiber of the other projection $q\colon E \rightarrow \mathbb{P}^1$, then $f \sim \ell$ in $X$, which implies that $N_\mathcal{F} \cdot f = 0$ for every fiber of $q$. This implies that every fiber of $q$ is tangent to $\mathcal{F}$ as well (see remark \ref{rmk:intersectionwithrationalcurve}). We conclude that $E$ is invariant by $\mathcal{F}$, a contradiction.
\end{proof}

In the previous proof, we used the following lemma:

\begin{lem}[{{\cite[Lemma 3.5.3.1]{brunella}}}]\label{lem:hirzsurface}
Let $X$ be a projective manifold and let $S \subset X$ be a Hirzebruch surface. Suppose $\mathcal{F}$ is a regular foliation on $X$. Then $\mathcal{F}_{|S}$ is induced by a $\mathbb{P}^1$-bundle. 
\end{lem}
\begin{proof}
Let $F$ be a fiber of $S \rightarrow \mathbb{P}^1$ and $C$ a minimal section, so that $F$ and $C$ generate ${\rm Pic}(S)$. Write $D$ for the class of $(N_{\mathcal{F}})_{|S}$. Since $\mathcal{F}$ is regular, we have $N_\mathcal{F}^2 \equiv 0$ and for all rational curves $\ell \subset X$, $N_\mathcal{F} \cdot \ell \geq 0$ with equality iff $\ell$ is tangent to $\mathcal{F}$ (see remark \ref{rmk:intersectionwithrationalcurve}).

Write $D = aF + bC$, for $a,b \in \mathbb{Z}$, and $n = -C^2$. Then $D^2 = 2ab - b^2n$, $D \cdot F = b$ and $D \cdot C = a - bn$. We thus see that $2ab -b^2n = 0$, $b \geq 0$ and $a-bn \geq 0$. Now, suppose that $F$ is not tangent to $\mathcal{F}$. Then $b > 0$, and we may simplify the first equation to $2a-bn =0$. This together with $a-bn \geq 0$ entails $bn \leq 0$, which can only happen if $n = 0$. In particular $a = 0$ and $D \cdot C = 0$, implying that $C$ is tangent to $\mathcal{F}$. Since $S = \mathbb{F}_0$, every other fiber of the second projection, linearly equivalent to $C$, has zero intersection with $D$ as well, and thus is tangent to $\mathcal{F}$. Therefore, $\mathcal{F}$ is given by the second projection $S \rightarrow \mathbb{P}^1$. If $F$ is tangent to $\mathcal{F}$, the same argument shows that $\mathcal{F}$ is given by the original projection $S \rightarrow \mathbb{P}^1$. This concludes the proof.
\end{proof}

Finally, the next lemma implies that in the case (2b) of theorem \ref{thm:mori}, a regular foliation of codimension one is always induced by the Mori Fiber Space.

\begin{lem}\label{lem:fibrationbysimplyconnectedhypersurfaces}
Let $Y$ be a projective manifold of dimension $3$ and let $C$ be a smooth curve. Suppose $f\colon Y \rightarrow C$ is a fibration with $\rho(Y/C) = 1$ and for the general fiber $F$ of $f$, we have ${\rm Pic}(F)$ torsion free and irregularity $q(F) = 0$. Let $\mathcal{G}$ be a codimension $1$ regular foliation on $X$. Then $\mathcal{G}$ is induced by $f$
\end{lem}
\begin{proof}
 Let $F$ be a fiber of $\pi$. Let $\omega \in H^0(X,\Omega_X^1 \otimes (N_\mathcal{G}))$ define $\mathcal{G}$ and let $\omega_{|F} \in H^0(F,\Omega_F^1 \otimes (N_\mathcal{G})_{|F})$ be induced by $F \hookrightarrow X$. Since $\rho(X/Y) = 1$, we have $(N_\mathcal{G})_{|F}$ ample, or $(N_\mathcal{G})_{|F}^*$ ample, or $(N_\mathcal{G})_{|F} \equiv 0$. By theorem \ref{thm:baumbott}, we have $N_{\mathcal{G}}^2 \equiv 0$, and thus, since $\dim(F) = 2$, we can only have $(N_\mathcal{G})_{|F} \equiv 0$. This shows that $(N_\mathcal{G})_{|F}$ is a torsion element of ${\rm Pic}(F)$. Since ${\rm Pic}(F)$ torsion-free, it follows that $(N_\mathcal{G})_{|F} \sim 0$. This implies that $\omega_{|F} \in H^0(F,\Omega_F^1)$, which is zero by $q(F) = 0$. We conclude that $\omega_{|F} \equiv 0$, showing that $F$ is invariant by $\mathcal{G}$. Therefore, the general fiber of $f$ is invariant by $\mathcal{G}$, and thus $\mathcal{G}$ is induced by $f$.
\end{proof}

\section{Proof of main result}
Using results of the previous section, in this section we will prove theorem \ref{thm:main} (which will follow from corollaries \ref{cor1} and \ref{cor2}). Let $X$ be a rationally connected threefold with $-K_X$ nef, and let $\mathcal{F}$ be a regular foliation of codimension $1$ on $X$. Theorem \ref{thm:main} states that $\mathcal{F}$ is induced by a smooth morphism with rational fibers.

To prove this, we run a foliated MMP according to theorem \ref{thm:spicer}. Let $(Y,\mathcal{G})$ be the end result. There are two possibilites: either $K_\mathcal{F}$ is pseudo-effective, or it is not. Firt we consider the latter case, namely if $K_\mathcal{F}$ is not pseudo-effective, then theorem \ref{thm:spicer} says that $Y$ is a fibration $\pi\colon Y \rightarrow B$ in one of the cases in (2) of theorem \ref{thm:mori}, and its fibers are tangent to $\mathcal{G}$. In particular, these fibers have dimension one or two. If $\varphi$ is a del Pezzo fibration, then it follows that $\mathcal{G}$ is induced by $\varphi$. 

The case of conic bundle will follow from the following lemma:

\begin{lem}\label{lem:conicbundlecase}
Let $\pi\colon Y \rightarrow S$ be a conic bundle over a surface as in theorem \ref{thm:mori}. Let $\mathcal{H}$ be foliation of rank $1$ on $S$, such that the pulled back foliation $\pi^{-1}(\mathcal{H})$ is regular. Then $\mathcal{H}$ is a regular foliation on $S$. 
\end{lem}
\begin{proof}
Let $\Delta \subset S$ be the set of points over which $\pi$ is not smooth. Then by theorem \ref{thm:mori}, outside the singular set of $\Delta$, the fiber of $\pi$ is always reduced. In particular, outside the set of reduced fibers, the non-smooth locus of $\pi$ has codimension at least $2$: it consists of the singular points of the singular fibers of $\pi$. Let $p \in S$ be any point. Then there exists an analytic neighborhood $U$ of $p$ in $S$, and a holomorphic $1$-form $\omega$ on $U$, such that $\mathcal{H}$ is induced by $\omega$. Suppose $\omega$ vanishes only at $p$. Then $\pi^*(\omega)$ vanishes along $\pi^{-1}(p)$. Moreover, it can vanish at points contained in the locus where $\pi$ is not smooth. However, since this locus has codimension at least $2$, we conclude that $\pi^*(\omega)$ vanishes along a set of codimension $2$. In particular, this $1$-form defines $\pi^{-1}(\mathcal{H})$ in a neighborhood of every point in $\pi^{-1}(p)$. But this would imply that $\pi^{-1}(\mathcal{H})$ is singular along $\pi^{-1}(p)$, a contradiction. Thus $\mathcal{H}$ has to be regular.
\end{proof}

In our context, lemma \ref{lem:pullbackbyequidimensional} implies that $\mathcal{G} = \pi^{-1}(\mathcal{H})$ for some foliation $\mathcal{H}$ of rank one on $S$. Since $\mathcal{G}$ is regular, we have by lemma \ref{lem:conicbundlecase} that the foliation $\mathcal{H}$ is regular. Finally, theorem \ref{thm:brunella} implies that $\mathcal{H}$ is induced by a smooth morphism with rational fibers. This shows that $\mathcal{G}$, and hence $\mathcal{F}$, is induced by a morphism with rational general fiber. Thus, when $K_\mathcal{F}$ is not pseudo-effective we have the following:

\begin{cor}\label{cor1}
Let $X$ be a rationally connected threefold. Let $\mathcal{F}$ be a regular foliation of codimension $1$ on $X$, with $K_\mathcal{F}$ non pseudo-effective. Then $\mathcal{F}$ is induced by a smooth morphism with rational fibers.
\end{cor}
\begin{proof}
We only need to show that the morphism $\varphi\colon X \rightarrow C$ inducing $\mathcal{F}$ is smooth. Indeed, since $X$ is rationally connected, we must have $C \cong \mathbb{P}^1$. Since the foliation induced by $\varphi$ is smooth, if $F$ is any fiber of $\varphi$, then there exists a smooth surface $F'$ such that $F \sim m F'$, for some $m > 0$. By \cite{graber2003families}, $\varphi$ admits a section $\ell$. In particular, $\ell \cdot F = 1$, and we conclude that $m = 1$, for any fiber $F = mF'$. Thus $\varphi$ is smooth.
\end{proof}

For the rest of the paper, we will suppose that $K_\mathcal{F}$ is pseudo-effective (our goal is to get a contradiction). Then $K_\mathcal{G}$ is nef. Now we are going to use the hypothesis $-K_X$ nef. The first step is to show that $-K_Y$ is nef as well. The two following lemmas will ensure this.

\begin{lem}
Let $X$ be a projective manifold of dimension $3$. Let $A$ be a nef divisor and $B$ a pseff divisor on $X$. Suppose $H\cdot B^2 \geq 0$ for some ample divisor $H$ on $X$. If $(A+B)^2 \equiv 0$, then $\alpha A\equiv \lambda B$ for some $\alpha,\lambda \in \mathbb{R}$, not both zero. \label{lem:square}
\end{lem}
\begin{proof}
	Let $H$ be the ample divisor in the statement. Then $H\cdot(A+B)^2 = 0$, i.e. $H\cdot A^2 + H\cdot B^2 + 2H\cdot A \cdot B = 0$. This is a sum of three non-negative numbers, and thus each one has to be zero. By the Lefschetz decomposition (\cite[page 122]{griffiths1978principles}), $N^1(X) = \mathbb{R}[H] \oplus P^2(X)$, where $P^2(X)$ consists of divisor classes satisfying $C\cdot H^2 = 0$. By the Hodge-Riemann bilinear relations (\cite[page 123]{griffiths1978principles}), for every non-zero divisor $C$ with class in $P^2(X)$, $H\cdot C^2 > 0$. If $W = \langle A, B \rangle$ has dimension $2$, then there is a non-zero element $C$ of $P^2(X)$ in $W$. On the other hand, for all $C \in W$, we have $H\cdot C^2 = 0$, and we get a contradiction. Thus $W$ has dimension $\leq 1$.
\end{proof}

\begin{lem}
Let $X$ and $Y$ be projective manifolds of dimension $3$. Let $\varphi\colon X \rightarrow Y$ be a smooth blowup of a curve $C$. Then $K_Y^2 = \varphi_*(K_X^2) + C$ in $A_1(Y)$. \label{lem:positivesquare}
\end{lem}
\begin{proof}
Let $E$ be the exceptional divisor of $\varphi$. Then $K_X = \varphi^*(K_Y) + E$ and thus $K_Y^2 = \varphi_*(\varphi^*(K_Y)\cdot K_X) = \varphi_*((K_X - E) \cdot K_X) = \varphi_*(K_X^2 - E \cdot K_X)$. Write $(K_X)_{|E} = af + b\tilde{C}$, where $f$ is a fiber and $\tilde{C}$ a section of $\varphi_{|E}\colon E \rightarrow C$. Since $\varphi$ is a blow up, $(K_X)_{|E}\cdot f = -1$, and thus $(K_X)_{|E} = af - \tilde{C}$. We conclude then that $\varphi_*(-E\cdot K_X) = (\varphi_{|E})_*(-(K_X)_{|E}) = (\varphi_{|E})_*(-af+\tilde{C}) = (\varphi_{|E})_*(\tilde{C}) = C$, and thus $\varphi_*(-E\cdot K_X) = C$. This shows that $K_Y^2 = \varphi_*(K_X^2)+C$ in $A_1(Y)$. 
 Let $H'$ be an ample divisor on $Y$. Then $H = \varphi^*(H')$ is a nef divisor on $X$ and $H'\cdot K_Y^2 = H \cdot K_X^2 - H \cdot E \cdot K_X$. By adjunction, $(K_X)_{|E} = K_E - E_{|E}$, and thus $H \cdot E \cdot K_X = H_{|E} \cdot (K_X)_{|E} = H_{|E} \cdot K_E - H_{|E} \cdot E_{|E}$. Now, $E = \mathbb{P}_C(N_{C/Y})$; denote by $f$ a fiber of $E \rightarrow C$ and by $h$ a section satisfying $h^2 = 0$ and $h\cdot f = 1$. Then $K_E = (2g(C) - 2) f - 2h$. Since $H = \varphi^*(H')$, we have $H_{|E} \cdot f = 0$, which implies that $H_{|E} \cdot K_E = -2 H_{|E} \cdot h \leq 0$. Since $-E_{|E}$ is relatively ample with respect to $\varphi$, by the same argument, $H_{|E}\cdot E_{|E} < 0$. We conclude in this way that $H' \cdot K_Y^2 > 0$.
\end{proof}

These two lemmas imply the following proposition:

\begin{prop}\label{thm:-kxpseff}
Let $X$ be a smooth rationally connected projective threefold. Suppose $-K_X$ is nef. Let $\mathcal{F}$ be a regular codimension $1$ foliation on $X$, with $K_\mathcal{F}$ pseudo-effective. Let $(Y,\mathcal{G})$ be the end result of a foliated MMP according to theorem \ref{thm:spicer}. Then $-K_Y$ is nef and $K_Y^2 \equiv 0$. Moreover, $Y$ is in cases (1) or (2a) of theorem \ref{thm:mori}. \label{prop:main}
\end{prop}
\begin{proof}
By theorem \ref{thm:spicer}, $\varphi\colon X \rightarrow Y$ is a composition of smooth blow-ups centered at smooth curves. Suppose for simplicity that $\varphi$ is a single centered at a smooth curve $C$. The general case follows from this one by an inductive argument. Let $H$ be an ample divisor in $Y$. Then, by lemma \ref{lem:positivesquare}, 
\[H\cdot K_Y^2 = H \cdot (\varphi_*(K_X^2) + C) = \varphi^*(H)\cdot K_X^2 + H\cdot C. \]
By hypothesis, $H'\cdot K_X^2 \geq 0$ for every ample divisor $H'$ on $X$. Thus, the same is true for every nef divisor. We conclude that $H\cdot K_Y^2 \geq 0$, for every ample divisor $H$ on $Y$.

Now, $N_\mathcal{G} = K_\mathcal{G} - K_Y$. Moreover, since $\mathcal{G}$ is regular, $N_\mathcal{G}^2 \equiv 0$ by theorem \ref{thm:baumbott}. Thus, we have $K_\mathcal{G}$ nef, $-K_Y$ pseff and $H\cdot K_Y^2 \geq 0$ for every ample divisor $H$ on $Y$, and it follows from lemma \ref{lem:square} that there exist $\alpha,\lambda \in \mathbb{R}$, non both zero, such that $\alpha K_\mathcal{G} \equiv \lambda K_Y$. If $\alpha = 0$, then $K_Y \equiv 0$, contradicting the fact that $Y$ is rationally connected. We may then suppose that $\alpha = 1$. Moreover, $K_Y$ is not nef, while $K_\mathcal{G}$ is; this implies that $\lambda < 0$. In particular, $-K_Y$ is nef. We also have $N_\mathcal{G} = (\lambda - 1)K_Y$, and since $N_\mathcal{G} \neq 0$, we have $K_Y^2 \equiv 0$.

Let us finish by showing that  $Y$ is in cases (1) or (2a) of theorem \ref{thm:mori}. First, $Y$ cannot be Fano, because $K_Y^2 \equiv 0$. In this case then, there exists an extremal negative contraction $\pi\colon Y \rightarrow Z$ (classified by theorem \ref{thm:mori}). If $\dim(Z) = 1$, then $\pi$ is a del Pezzo fibration. By lemma \ref{lem:fibrationbysimplyconnectedhypersurfaces}, $\mathcal{G}$ is induced by $\pi$. But in this case $(K_\mathcal{G})_{|F} \equiv K_F$, for a general fiber $F$ of $\pi$, a contradiction to the fact that $F$ is rational and $K_\mathcal{G}$ is nef.

Thus, we conclude that the only remaining possibilities for $Y$ in theorem \ref{thm:mori} are the ones in (1) and (2a). 
\end{proof}

Thus we may take the nef reduction map $f\colon Y \rightarrow \mathbb{P}^1$ associated to $-K_Y$ according to theorem \ref{thm:3fold}. By proposition \ref{prop:main} and theorem \ref{thm:dimnum1}, we conclude that $f$ is a $K3$-fibration, and $-K_Y \cong f^*(\mathcal{O}_{\mathbb{P}^1}(1))$. We first show that in this case,  $\mathcal{G}$ is necessarily induced by $f$.

\begin{lem}
Let $f \colon Y \rightarrow \mathbb{P}^1$ a $K3$-fibration induced by $-mK_Y$, for some $m > 0$. Suppose $\mathcal{G}$ is a codimension one regular foliation on $Y$ such that $N_\mathcal{G} \equiv \alpha K_Y$, for some $\alpha$. Then $\mathcal{G}$ is induced by $f$.
\end{lem}
\begin{proof}
Let $\omega \in H^0(Y,\Omega_Y^1 \otimes N_\mathcal{G})$ define $\mathcal{G}$. Let $F$ be a general fiber of $f$. Then $N_\mathcal{G} \cdot F = 0$, because $f$ is given by $|-mK_Y|$ (hence $-mK_Y\cdot F = 0$). Since $F$ is K3, we have $\pi_1(F) = 0$, and this implies that $(N_\mathcal{G})_{|F} \sim 0$. Thus $\omega_{|F} \in H^0(F,\Omega_F^1)$, and this group is $0$, again by $\pi_1(F) = 0$. We conclude that $\omega_{|F} \equiv 0$, implying that $F$ is tangent to $\mathcal{G}$, for the general $F$, fiber of $f$. This shows that $\mathcal{G}$ is induced by $f$. 
\end{proof}

Finally, we get a contradiction by showing that in this case, since $\mathcal{G}$ is regular, $f$ has to be smooth and hence isotrivial.

\begin{lem}\label{lem:k3fibration}
Let $Y$ be a rationally connected manifold of dimension $3$. Suppose there is a fibration $f \colon Y \rightarrow \mathbb{P}^1$, whose general fiber has numerically trivial canonical bundle, such that $-K_Y = f^*(\mathcal{O}_{\mathbb{P}^1}(1))$. If $\mathcal{G}$ is the foliation induced by $f$, then ${\rm sing}(\mathcal{G}) \neq \emptyset$.
\end{lem}
\begin{proof}
Suppose that ${\rm sing(\mathcal{G})} = \emptyset$. Let us consider all the cases for $Y$ in theorem \ref{thm:mori}. First, since $-K_Y = f^*(\mathcal{O}_\mathbb{P}^1(1))$, $Y$ cannot be Fano. If $Y$ admits a fibration by del Pezzo surfaces, with relative Picard number equal to one, then by lemma \ref{lem:fibrationbysimplyconnectedhypersurfaces}, any regular foliation of codimension one on $Y$ is induced by this fibration. If $Y$ is a conic bundle as in theorem \ref{thm:mori}, with fibers tangent to $\mathcal{G}$, then by lemma \ref{lem:conicbundlecase} and theorem \ref{thm:brunella}, any regular foliation of codimension one on $Y$ has rational leaves. Thus, in all these cases we would contradict the hypothesis that the leaves of $\mathcal{G}$ have numerically trivial canonical bundle, and hence they cannot happen.

It remains to treat the cases of divisorial contraction and conic bundle whose fibers are generically transverse to $\mathcal{G}$. By lemma \ref{lem:divtopoint}, if $Y$ admits a divisorial contraction as in theorem \ref{thm:mori}, then the only possible case is that it is a smooth blow-up of a curve. 

Since we are supposing that the foliation induced by $f$ is regular, if $F'$ is a singular fiber of $f$, then there exists a regular surface $F'' \subset Y$, and a positive integer $m > 0$, such that $F' = mF''$. Moreover, since $-K_Y = f^*(\mathcal{O}_{\mathbb{P}^1}(1))$, we have $-K_Y \sim F'$. Thus $K_Y \cdot F' \equiv 0$ and $F'\cdot F' \equiv 0$ (since it is a fiber), and we conclude that $K_Y \cdot F'' \equiv 0$ and $F'' \cdot F'' \equiv 0$. By the adjunction formula, $K_{F''} \equiv 0$.

If $Y$ is the blow-up of a smooth $3$-fold along a curve, then taking $\ell$ a rational curve contracted by it, we have $-K_Y \cdot \ell = 1$. This implies that, for any singular fiber $F' = mF''$ of $f$, we have $mF'' \cdot \ell = 1$, and thus $m = 1$. We conclude, in this case, that $f$ is smooth.

If $\varphi\colon Y \rightarrow S$ is a conic bundle with fibers generically transverse to $\mathcal{F}$, then for any such fiber $\ell$, we have $-K_Y \cdot \ell = 2$. Thus, for any fiber $F' = mF''$, we have $mF'' \cdot \ell = 2$, and hence $m = 1$ or $m = 2$. If $m = 1$ for all fibers, then again $f$ is smooth. Suppose then that $m = 2$ for some fiber $F'$. Thus, in this case, $F'' \cdot \ell = 1$. This implies that the restriction $\varphi_{|F''} \colon F'' \rightarrow S$ is an isomorphism. Since $S$ is rational, this contradicts the fact that $K_{F''} \equiv 0$.

We conclude that in both cases, $f$ has to be smooth. By theorem \ref{thm:viehwegzuo}, $f$ is birationally isotrivial. Thus, by lemma \ref{lem:oguisoviehweg}, since $\mathbb{P}^1$ is simply connected, we conclude that $f$ is trivial, in other words, $Y \cong F \times \mathbb{P}^1$, for $F$ a fiber of $f$, and $f$ is the projection to $\mathbb{P}^1$. Since $Y$ is rationally connected, this implies that $F$ is rational, a contradiction to the fact that $K_F \equiv 0$. Thus, no case of theorem \ref{thm:mori} is possible when $\mathcal{G}$ is regular, which shows that ${\rm sing}(\mathcal{G}) \neq \emptyset$.
\end{proof}

The contradiction to the fact that $\mathcal{G}$ is regular obtained from the last lemma follows from the assumption that $K_\mathcal{G}$ is pseudo-effective. We thus conclude the following, which is the final step to show theorem \ref{thm:main}.

\begin{cor}\label{cor2}
Let $X$ be a rationally connected threefold with $-K_X$ nef. Let $\mathcal{F}$ be a regular foliation of codimension $1$ on $X$. Then $K_\mathcal{F}$ is not pseudo-effective.
\end{cor}

\begin{proof}[Proof of theorem \ref{thm:main}]
Let $X$ be a rationally connected threefold with $-K_X$ nef and let $\mathcal{F}$ be a codimension $1$ regular foliation on $X$. Then, by corollary \ref{cor2}, $K_\mathcal{F}$ is not pseudo-effective, which implies, by corollary \ref{cor1}, that $\mathcal{F}$ is induced by a smooth morphism with rational fibers.
\end{proof}

\bibliographystyle{alpha}
\bibliography{main}
\end{document}